\documentclass[a4paper]{article}

\usepackage[utf8]{inputenc}
\usepackage[T1]{fontenc}

\usepackage{amsmath, amsthm}
\usepackage{amssymb}
\usepackage{hyperref}
\usepackage{enumerate}
\usepackage{cite}
\usepackage{cuted}
\usepackage{mathrsfs}

\newtheorem{defi}{Definition}[section]
\newtheorem{theorem}[defi]{Theorem}
\newtheorem{prop}[defi]{Proposition}

\newtheorem{cor}[defi]{Corollary}
\newtheorem{remark}[defi]{Remark}
\theoremstyle{definition}

\newcommand{\R}{\mathbb{R}}
\newcommand{\C}{\mathbb{C}}

\newcommand{\F}{\mathbb{F}}

\newcommand{\ran}{{\rm ran}\,}
\DeclareMathOperator{\re}{Re}

\title{Linear-quadratic optimal control for abstract differential-algebraic equations}

\author{H.\ Gernandt \thanks{Fraunhofer IEG, Fraunhofer Research Institution for Energy Infrastructures and Geothermal Systems IEG, Cottbus, Gulbener Straße 23, 03046 Cottbus, Germany and Universit\"{a}t Wuppertal, 
Gau\ss stra\ss e 20, 42119 Wuppertal (e-mail: gernandt@uni-wuppertal.de)} \and T.\ Reis \thanks{TU Ilmenau, Weimarer Straße 25, 98693 Ilmenau, Germany (e-mail: timo.reis@tu-ilmenau.de)}}

\begin{document}

\maketitle

\begin{abstract}
In this paper, we extend a~classical approach to linear quadratic (LQ) optimal control via Popov operators to abstract linear differential-algebraic equations (ADAEs) in Hilbert spaces. To ensure existence of solutions, we assume that the underlying differential-algebraic equation has index one in the pseudo-resolvent sense. This leads to the existence of a degenerate semigroup that can be used to define a Popov operator for our system. 
It is shown that under a suitable coercivity assumption for the Popov operator the optimal costs can be described by a bounded Riccati operator and that the optimal control input is of feedback form. Furthermore, we characterize exponential stability of ADAEs which is required to solve the infinite horizon LQ problem.
\end{abstract}
\textit{Keywords:} optimal control, differential-algebraic equations, infinite dimensional systems, linear quadratic regulator, stability \vspace{0.5cm}\\

\textit{MSC 2010:} 
	34A09, 34G10, 49N05

\section{Introduction} 
In this note, we consider abstract differential-algebraic systems in Hilbert spaces of the form
\begin{equation}
\begin{split}
\label{system}
\tfrac{{\rm d}}{{\rm d}t}Ex(t)&=Ax(t)+Bu(t),\quad Ex(0)=Ex_0,\\ y(t)&=Cx(t),\quad t\geq 0,
\end{split}
\end{equation}
where $E\in L(X,Z)$, $C\in L(X,Y)$ and $B\in L(U,Z)$ are bounded linear operators and $A: X\supseteq D(A)\rightarrow Z$ is closed and densely defined and $X$, $Z$, $U$ and $Y$ denote the Hilbert spaces of states, images of states, inputs and outputs, respectively. 

The specific challenge arises from the fact that, in numerous practically motivated instances, as observed in \cite{MehrZwar23,Reis07,Reis06,ReisTisc05}, the operator $E$ possesses non-trivial kernel and co-kernel. Examples of such scenarios include abstract Cauchy problems incorporating additional closure relations, physical problems characterized by vanishing material parameters or densities, or systems of coupled partial differential equations (PDEs) like heat-wave couplings formulated in an abstract differential-algebraic form.
Subject to the dynamics \eqref{system}, our objective is to minimize the quadratic cost on some time horizon $t_f>0$ given by 
\begin{align}
\label{quadcost}
J(u, y) = \int_0^{t_f} \left<\begin{pmatrix}y(t)\\ u(t)\end{pmatrix},\begin{bmatrix}\mathcal{Q}(t) & \mathcal{N}^*(t) \\ \mathcal{N}(t) & \mathcal{R}(t)\end{bmatrix}\begin{pmatrix}y(t)\\ u(t)\end{pmatrix} \right>_{Y\times U} {\rm d}t
\end{align}
where $\mathcal{Q}\in L^\infty([0,t_f],L(Y))$, $\mathcal{R}\in L^\infty([0,t_f],L(U))$, $\mathcal{N}\in L^\infty([0,t_f],L(Y,U))$ (that is, $\mathcal{Q}$, $\mathcal{R}$ and $\mathcal{N}$ are bounded functions with values in the space of linear operators), and additionally, $Q$ and $R$ are pointwise self-adjoint. This minimization is performed over all solutions of \eqref{system} with $u\in L^2([0,t_f],U)$ and $y \in L^2([0,t_f],Y)$.

As a main assumption, we use that the abstract differential-algebraic system has index at most one and is stable in the pseudo-resolvent sense, meaning that the following holds
\begin{equation}\label{eq:resgrowth}
\exists\, M, \omega>0:\quad\|(A-\lambda E)^{-1}E\|\leq \frac{M}{\lambda-\omega} \quad \text{ $\forall\,\lambda>\omega$.}    
\end{equation}
The incorporation of this index condition, alongside an additional stability assumption, enables us to relate our setup to the optimal control framework for standard infinite-dimensional systems presented in \cite{WeisWeis97}.

Further, we assume that the {\em transfer function} $G(s)=C(sE-A)^{-1}B$ is bounded in the complex right half-plane, that is
\begin{equation}
    C(\cdot E-A)^{-1}B\in H^{\infty}.\label{eq:GHinf}
\end{equation}
This condition guarantees that the output of the system \eqref{system} depends continuously on the input in the $L^2$-sense. We will further discuss how to eliminate these assumptions towards the conclusion of this article.

Our findings are inspired by previous results for finite-dimensional differential-algebraic equations \cite{ReisVoig19}, see also \cite{Mehr91} for an overview in the ODE case. Furthermore, we use the solution theory that is developed for inhomogeneous ADAEs of finite pseudo-resolvent index \cite{GernReis23} to define Popov operators for ADAEs. These operators can be used so solve the LQ problem on finite-time and infinite time-horizons. Recently, a different approach to LQ optimal control of ADAEs with radiality index was presented in \cite{AlaM24} using the solutions of differential Riccati equations, see also \cite{ErbJMRT24,GernReis23} for a comparison of the different index concepts.

In this note, we consider also the infinite horizon LQ problem, and here it is necessary to have exponential stability of the underlying degenerate semi-groups. Therefore, we also provide Lyapunov-like characterizations of exponential stability in terms of the pseudo-resolvents and in terms of the coefficients of the ADAE. 

The paper is organized as follows: in section~\ref{sec:prelim} we recall some preliminaries on pseudo-resolvents, mild solutions and degenerate semigroups. In section~\ref{sec:stability} we characterize the exponential stability of abstract differential-algebraic equations having pseudo-resolvent index at most one. The previous results are used to provide a solution to the linear quadratic optimal control problem in section~\ref{sec:lqr} using Popov operators.

\section{Preliminaries}
\label{sec:prelim}
In this section, we recall some results from \cite{GernReis23} on solvability of abstract differential-algebraic equations of the form 
\begin{equation}
\label{adae}
\tfrac{{\rm d}}{{\rm d}t}Ex(t)=Ax(t)+f(t),\quad (Ex)(0)=Ex_0
\end{equation}
with $f\in L^2[0,t_f]$ for some $t_f>0$.  We assume that the pair $(E,A)$ is regular, i.e.\ we have a non-empty \textit{resolvent set}
\[
\rho(E,A):=\{\lambda\in\C\, :\, \text{$\lambda E-A$ bijective}\}.
\]

\subsection{Pseudo-Resolvents}

For all $\lambda,\mu\in\rho(E,A)$, we have the resolvent identity
\begin{align}
\label{residpencil}
&\,\,\,\,\,\,\,\,(\mu-\lambda)(A-\mu E)^{-1}E(A-\lambda E)^{-1}\nonumber\\&=(A-\lambda E)^{-1}-(A-\mu E)^{-1}\\&=(\mu-\lambda)(A-\lambda E)^{-1}E(A-\mu E)^{-1}. \nonumber
\end{align}
Multiplying \eqref{residpencil} by $E$ from both the left and the right implies that 
\begin{align}
    R_l(\lambda)&:=E(A-\lambda E)^{-1},\label{eq:Rl}\\ R_r(\lambda)&:=(A-\lambda E)^{-1}E\label{eq:Rr}
\end{align} are what are known as \emph{left and right pseudo-resolvents}, see \cite{GernReis23}. The openness of $\Omega:=\rho(E,A)$, combined with the identity above, implies that both $R_l$ and $R_r$ are holomorphic operator-valued functions.

If a pseudo-resolvent $R:\Omega\rightarrow L(X)$ with $(\omega,\infty)\subseteq \Omega$ for some $\omega\in\R$ satisfies a growth condition 
\begin{equation}
\label{arendtcond}
\|R(\lambda)\|\leq \frac{M}{\lambda-\omega}, \quad \lambda\in \Omega\cap\R\text{ with }\lambda>\omega ,
\end{equation}
then by a generalization of the Hille-Yosida theorem it generates a so-called {\em degenerate semigroup}, see \cite{Aren01, GernReis23}. 
This concept will be treated in the next section.

We first show that pseudo-resolvents with property \eqref{arendtcond} also admit a~decomposition of $X$, see e.g.\ \cite{Kato}.
\begin{prop}
\label{prop:arendtyagi}
Let $X$ be a Hilbert space and $R:\Omega\rightarrow L(X)$ be a pseudo-resolvent which satisfies \eqref{arendtcond}. Then for all $\lambda\in\Omega$, $X$ decomposes into the direct sum
\begin{equation}
X=X_K\dotplus X_R= \ker R(\lambda)\dotplus \overline{\ran R(\lambda)}, \label{eq:Xdecomp}
\end{equation}
and this decomposition does not depend on the choice of the evaluation point, i.e.\ it holds that $\ran R(\lambda)=\ran R(\mu)$ and $\ker R(\lambda)=\ker R(\mu)$ for all $\mu\in\Omega$.
\end{prop}

It is shown in \cite[Theorem 3.6]{JacoMorr22}, that the growth bound \eqref{arendtcond} follows from dissipativity conditions on the coefficients $E$ and $A$, see also \cite[Section 7]{GernReis23} for a possible extension to a~Banach space setting and a~more general dissipativity concept. 
\begin{theorem}
\label{thm:jacob_morris}
Let $\lambda\in\rho(E,A)$ for some $\lambda>0$ and assume that
\begin{align*}
\re \langle Ax, Ex\rangle_Z &\leq  0,\, x \in D(A),\\
\re \langle A^*x, E^*x\rangle_X &\leq  0, x \in D(A^*).
\end{align*}
Then the following holds
\begin{itemize}
    \item[1)] $(0,\infty)\subseteq\rho(E,A)$ and $(0,\infty)\subseteq\rho(E^*,A^*)$ 
    \item[2)] $\|E(\lambda E - A)^{-1}\| \leq  \lambda^{-1}$ for $\lambda>0$;
    \item[3)] $\|(\lambda E - A)^{-1}E\| \leq  \lambda^{-1}$ for $\lambda>0$.
\end{itemize}
\end{theorem}

\subsection{Degenerate semigroups}
Here we introduce degenerate semigroups and present some facts on their relation to pseudo-resolvents. A~strongly continuous mapping $T:(0,\infty)\rightarrow L(X)$ is called \emph{degenerate semigroup} if 
\[
T(t+s)=T(t)T(s),\quad s,t>0, \quad \sup_{0<t\leq 1}\|T(t)\|<\infty.
\]
Since $X$ is a Hilbert space, it is reflexive and this implies, see \cite[Corollary 2.2]{Aren01} that the following strong limit exists $T(0):=\lim\limits_{t\downarrow 0}T(t)$. Degenerate semigroups with this property are called \emph{strongly continuous}. In this case $T(0)$ is a~bounded projection onto  $X_R:=\ran T(0)$. 
The restriction of $T$ to $X_R$, i.e., $(T(t)|_{X_R})_{t\geq 0}$, is moreover a~strongly continuous semigroup on $X_R$. Further, we have that $\ker T(t)=\ker T(0)$ for all $t\geq 0$.
In particular, 
\[\forall\,t\geq0:\quad(I-T(0))T(t)=0.\]
In other words, for $X_K:=\ker T(0)$, the operator-valued function $T(\cdot)|_{X_K}$ remains constantly zero.

For a~pseudo-resolvent $R$ that satisfies the growth condition \eqref{arendtcond} and $\{\lambda\in\C ~|~ \re\lambda> \omega\}\subseteq\Omega$, then the following subspace defines a closed densely defined linear operator $A_R$, which is uniquely defined by
\begin{multline}
\label{def:A_R}
{\rm graph}\,A_R\\=\{(R(\lambda)x,x+\lambda R(\lambda)x), x\in \overline{\ran R(\lambda)}\}\subseteq X_R\times X_R.
\end{multline}
This operator is indeed independent on the choice of $\lambda\in\Omega$, and it fulfills 
\[
\|(A_R-\lambda I)^{-1}\|\leq \frac{M}{\lambda-\omega} .
\]

Hence, according to \cite[Corollary 3.7.12]{ArenBatt11} it generates a $C_0$-semigroup $(T_R(t))_{t\geq 0}$ on $X_R$. Hence, we obtain a degenerate semigroup for $x=x_0+x_1\in X_K\dotplus X_R$
\[
T(t)x=T(t)(x_0+x_1)=T_R(t)x_1,\quad \text{for all $t\geq 0$.}
\]
Further, it holds that
\begin{align}
\label{commute}
\forall\, \lambda\in\Omega,\,t\geq 0:\quad R(\lambda)T(t)=T(t)R(\lambda).
\end{align}

\subsection{Exponential Stability of semigroups}

A strongly continuous semigroup $(T(t))_{t\geq 0}$ is called exponentially stable, if there exist $M,\omega>0$ such that 
\begin{equation}\forall\,t\geq0:\quad \|T(t)\|\leq M e^{-\omega t}.\label{eq:Texpstab}\end{equation}

We recall the following characterization of exponential stability of semigroups, see e.g.\ \cite[Theorem 8.1.3, Lemma 8.1.2, Theorem 8.1.4]{JacoZwar12}.
\begin{prop}
\label{prop:exp_stable}
Let $A$ be the infinitesimal generator of the $C_0$-semigroup $(T(t))_{t\geq0}$ on the Hilbert space $X$. Then the following is equivalent 
    \begin{itemize}
        \item[\rm (a)] $(T(t))_{t\geq0}$ is exponentially stable;
        \item[\rm (b)] There exists a positive operator $Q\in L(X)$ such that  
        \[
        \langle Ax,Qx\rangle+ \langle Qx,Ax\rangle=-\langle x,x\rangle,\quad \text{for all $x\in D(A)$.}
        \]
        \item[\rm (c)] There exists a positive operator $Q\in L(X)$ such that  
        \[
        \langle Ax,Qx\rangle+ \langle Qx,Ax\rangle\leq -\langle x,x\rangle,\quad \text{for all $x\in D(A)$.}
        \]
        \item[\rm (d)] For every $x\in X$ it holds $\int_0^\infty\|T(t)x\|^2{\rm d}t<\infty$.
        \item[\rm (e)] $(\cdot I-A)^{-1}\in H^{\infty}(L(X))$, that is, the mapping $\lambda\mapsto (\lambda I-A)^{-1}$ is bounded on the complex right half plane.
    \end{itemize}
\end{prop}
\begin{remark}
In a less stringent formulation, statement (b) refers to the existence of a positive operator $Q$ that satisfies the Lyapunov equation $A^*Q + AQ + I = 0$.
\end{remark}

\subsection{Mild Solutions}

Motivated by the study of linear quadratic optimal control problems, we consider here mild solutions of abstract differential-algebraic equations. Classical and weak solutions are discussed in further detail in \cite{GernReis23}. We call $x\in L^2([0,t_f],X)$ a~\textit{mild solution} of \eqref{adae} on $[0,t_f]$, $t_f> 0$, if $Ex:[0,t_f]\to Z$ is continuous, and 
\begin{align}
\forall\, t\in[0,t_f]:     \int_0^tx(\tau)d\tau&\in D(A) \label{eq:x_mild} \\
\text{ with }Ex(t)-Ex_0&=A\int_0^tx(\tau)d\tau+\int_0^tf(\tau)d\tau. \nonumber
\end{align}

To describe the mild solutions we consider the right pseudo-resolvent $R_r$ as in \eqref{eq:Rr},
and a~multiplication of \eqref{adae} from the left with $(A-\mu E)^{-1}$, $\mu\in\rho(E,A)$, and invoking
\begin{multline*}
    (A-\mu E)^{-1}A=(A-\mu E)^{-1}(A-\mu E+\mu E)\\=I-\mu(A-\mu E)^{-1}E,
\end{multline*}
we obtain, for $\hat f=(A-\lambda E)^{-1}f$
\begin{align*}    
\tfrac{{\rm d}}{{\rm d}t}R_r(\mu)x(t)&=(I-\mu R_r(\mu))x(t)+\hat f(t),\\ R_r(\mu)x(0)&=R_r(\mu)x_0.
\end{align*}
In particular, if $0\in\rho(E,A)$, we have, for $\hat f=A^{-1}f$ that
\begin{align}\label{eq:pseudoDGL}    
\frac{{\rm d}}{{\rm d}t}R_r(0)x(t)=x(t)+\hat f(t)\quad R_r(\mu)x(0)=R_r(\mu)x_0.
\end{align}

To describe the mild solutions explicitly, we decompose $\hat f=\hat f_R+\hat f_K\in L^2([0,t_f],X)$ according to the direct sum decomposition $X=X_K\dotplus X_R=\ker R(0)\dotplus\overline{\ran R(0)}$. 
Since $X_K$ and $X_R$ are both closed, there exists a~bounded projector $P\in L(X)$ that satisfies 
\begin{align}
\label{def:P}
\ran P=X_R,\quad \ker P=\ran(I-P)=X_K. 
\end{align}
In summary, we have
\begin{align}
    \label{eq:hatf_decomposition}
A^{-1}f=\hat f=\hat f_R+\hat f_K=P\hat f+(I-P)\hat f.
\end{align}
The mild solutions of the pseudo-resolvent equation \eqref{eq:pseudoDGL} have been studied in \cite[Lemma 8.1, Proposition 8.3]{GernReis23}. Below, we state the result only for the pseudo-resolvent $R_r(\lambda)=(A-\lambda E)^{-1}E$ and in the case that $\mu=0\in\rho(A)$.\\ 
We first present some facts from \cite{TucsWeis09} on unbounded control operators: For generator $D(A_R)\subset X_R\to X_R$ of a~strongly continuous semigroup $(T_R(t))_{t\geq 0}$, we can associate the space $X_{R,-1}$, which is obtained by the completion of $X$ with respect to the norm 
\[x\mapsto\|(\lambda I-A_R)^{-1}x\|_{X_R}\] for some $\lambda$ in the resolvent set of $A_R$. Indeed, by the resolvent identity, the topology induced by the above norm is independent on the choice of $\lambda$.\\
Further, $A_R$ extends to an operator $A_{R,-1}$, which generates a~semigroup $(T_{R,-1}(t))_{t\geq 0}$ on $X_{R,-1}$. The latter is also an extension of $(T_R(t))_{t\geq 0}$, and it it referred to as {\em extrapolated semigroup}. If $f_R:[0,t_f]\to X_{R,-1}$ is weakly differentiable, then the mild solution of $\tfrac{{\rm d}}{{\rm d}t}x_R(t)=A_Rx_R(t)+f_R(t)$, $x_R(0)=x_{R,0}\in X_R$ is given by $x:[0,t_f]\to X_R$ with
\[x_R(t)=T_R(t)x_0+\int_0^{t_f}T_{R,-1}(t-\tau)f_R(\tau){\rm d}\tau.\]
For a~degenerate semigroup $(T(t))_{t\geq 0}$ on $X$ with associated decomposition \eqref{eq:Xdecomp}, we define the {\em extrapolated degenerate semigroup} $(T_{-1}(t))_{t\geq 0}$ by 
\begin{align*}
    T_{-1}(t):\quad X_{R,-1}\dotplus X_K&\to X_{R,-1}\dotplus X_K,\\
    x_{0,R}+x_{0,K}&\mapsto T_{R,-1}(t)x_{0,R}.
\end{align*}

\begin{prop}
\label{prop:vdk}
Given the abstract differential-algebraic equation \eqref{adae} such that $R_r$ given by \eqref{eq:Rr} satisfies \eqref{eq:resgrowth} for some $\omega<0$ and let $T$ be the degenerated semigroup associated to $R(\lambda)=(A-\lambda E)^{-1}E$, and let $T_{-1}$ be the corresponding extrapolated semigroup. Let $A_R:D(A_R)\subset X_R\to X_R$ be the operator as in \eqref{def:A_R}. Further, we decompose $\hat f=\hat f_R+\hat f_K\in L^2([0,t_f],X)$ with $\hat f_R\in L^2([0,t_f],\overline{\ran R(0)})$ and $\hat f_K\in L^2([0,t_f],\ker R(0))$ as in \eqref{eq:hatf_decomposition}. Then for $x_0\in \overline{\ran R(0)}$ the unique mild solution of \eqref{adae} is given by 
    \begin{align}
    \label{eq:sol_x}
    x(t)=T_R(t)x_0+\int_0^tT_{R,-1}(t-\tau){A_R}\hat f_R(s)\,{\rm d}\tau-\hat f_K(t).
    \end{align}
\end{prop}
\begin{proof}
First, we show that $x$ given by \eqref{eq:sol_x} is the unique mild solution of the pseudo-resolvent equation~\eqref{eq:pseudoDGL}. To this end, we construct a mild solution $x_R$ on $X_R$ and $x_K$ on $X_K$.  
First, observe that $x_K(t):=-\hat f_K\in L^2([0,t_f],X_K)$ satisfies~\eqref{eq:psres_proof} pointwise for all $t\in[0,t_f]$. The mild solutions $x_R$ of \eqref{eq:psres_proof} for $x_R(0)=x_R^0\in X_R$  
\begin{align}
\label{eq:mild_R}
A^{-1}Ex_R(t)-A^{-1}Ex_R(0)=\int_0^t (x_R(\tau)+PA^{-1}f(\tau))\,{\rm d}\tau,
\end{align}
where $P$ is the projection onto $X_R$ that is defined in \eqref{def:P}. This can be rewritten using the operator $A_R$ from \eqref{def:A_R}, which is invertible and fulfills $A_R^{-1}=(A^{-1}E)|_{X_R}$. Therefore, \eqref{eq:mild_R} is equivalent to 
\begin{align}
\label{eq:mild_R_2}
x_R(t)-x_R(0)=A_R\int_0^t x_R(\tau)\,{\rm d}\tau +\int_0^t A_RPA^{-1}f(\tau)\,{\rm d}\tau.
\end{align}
By assumption $A_R$ generates the $C_0$-semigroup $\{T_R(t)\}_{t\geq 0}$, and hence, the unique mild solution $x_R$ of~\eqref{eq:mild_R_2} is given by
\[
x_R(t)=T_R(t)x_0+\int_0^tT_{R,-1}(t-s)A_RPA^{-1}f(\tau)\,{\rm d}s. 
\]
Next, we show that the mild solution of the pseudo-resolvent equation \eqref{eq:pseudoDGL} coincide with the mild solutions of \eqref{adae}. Let $A^{-1}Ex:[0,t_f)\rightarrow X$ be continuous and satisfying for all $t\geq 0$
\begin{align}
    \label{eq:psres_proof}
A^{-1}Ex(t)-A^{-1}Ex(0)&=\int_0^t x(\tau)\,{\rm d}\tau +\int_0^t\hat f(\tau)\,{\rm d}\tau\\
&=\int_0^t x(\tau)\,{\rm d}\tau +A^{-1}\int_0^tf(\tau)\,{\rm d}\tau.\nonumber 
\end{align}
Since $\ran A^{-1}=D(A)$, we conclude $\int_0^t x(\tau)\,{\rm d}\tau\in D(A)$. Therefore, applying $A$ to \eqref{eq:psres_proof} leads to \eqref{eq:x_mild} and therefore $x$ is a mild solution of \eqref{adae}. Conversely, if $x$ is a mild solution of \eqref{adae}, then $Ex$ is continuous and therefore $A^{-1}Ex$ is continuous and multiplying with $A^{-1}$ from the left in \eqref{eq:x_mild} implies \eqref{eq:psres_proof}. \hfill \hfill \qed
\end{proof}

\section{Exponential stability of abstract differential-algebraic equations}
\label{sec:stability}
Utilizing the generator $A_R$ as described in \eqref{def:A_R}, we can extend the characterization of exponential stability from Proposition~\ref{prop:exp_stable} to pseudo-resolvents and degenerate semigroups. In this context, for a~given degenerate semigroup $(T(t))_{t\geq 0}$ that is exponentially stable in the sense that it fulfills \eqref{eq:Texpstab}, we  apply Laplace transform yielding a~pseudo resolvent that is defined for all $\lambda\in\C$ with $\re(\lambda)>\omega$, as discussed in \cite{Aren01}, namely
\begin{align}
\label{def:R_from_T}
R(\lambda)x := \lim_{\tau \to \infty}\int_0^\tau e^{-\lambda t}T(t)x  {\rm d}t, \quad \re(\lambda) > \omega.
\end{align}
The following characterization of exponential stability can be directly proven by combining Proposition~\ref{prop:exp_stable} together with the decomposition \eqref{eq:Xdecomp}.
\begin{prop}
\label{prop:exp_stable_deg}
Let $(T(t))_{t\geq0}$ be a (continuous) degenerated semigroup on the Hilbert space $X$ with associated pseudo-resolvent $R:\Omega\rightarrow L(X)$ given by \eqref{def:R_from_T} with $0\in\Omega$. Then the following is equivalent.
    \begin{itemize}
        \item[\rm (a)] $(T(t))_{t\geq0}$ is exponentially stable.
        \item[\rm (b)] There exists a positive operator $Q\in L(X)$ such that  
        \[
        \langle R(0)x,Qx\rangle+ \langle Qx,R(0)x\rangle=-\langle R(0)x,R(0)x\rangle,\quad 
        \]
        for all $x\in \overline{\ran R(0)}$.
        \item[\rm (c)] There exists a positive operator $Q\in L(X)$ such that  
        \[
        \langle R(0)x,Qx\rangle+ \langle Qx,R(0)x\rangle\leq -\langle R(0)x,R(0)x\rangle,\quad 
        \]
        for all $x\in \overline{\ran R(0)}$.
        \item[\rm (d)] For every $x\in X$ it holds $\int_0^\infty\|T(t)x\|^2{\rm d}t<\infty$.
        \item[\rm (e)] $R\in H^{\infty}(L(X))$. 
    \end{itemize}
\end{prop}

Combining the latter results with Theorem~\ref{thm:jacob_morris}, we obtain the following result on exponential stability in the dissipative case. 
\begin{cor}
If $0\in\rho(A)$, then $0\in\rho(E,A)$ and assume that there exists $\omega>0$ such that
the following estimates hold
\begin{align}
\re \langle Ax, Ex\rangle_Z &\leq  -\omega\|Ex\|^2, x \in D(A),\\
\re \langle A^*x, E^*x\rangle_X &\leq  -\omega \|E^*x\|^2, x \in D(A^*),
\end{align}
Then we have 
\begin{itemize}
\item[1)]
$(-\omega,\infty)\subseteq\rho(E,A)$ and $(-\omega,\infty)\subseteq\rho(E^*,A^*)$;
    \item[2)] $\|E(\lambda E - A)^{-1}\| \leq  (\lambda+\omega)^{-1}$ for $\lambda>-\omega$;
    \item[3)] $\|(\lambda E - A)^{-1}E\| \leq  (\lambda+\omega)^{-1}$ for $\lambda>-\omega$.
\end{itemize}
In particular, the degenerated semigroup associated with the pseudo-resolvent $\lambda\mapsto (A-\lambda E)^{-1}E$ is exponentially stable. 
\end{cor}

\section{Linear quadratic optimal control of abstract differential-algebraic equations}
\label{sec:lqr}
In this section, we study LQ optimal control for ADAEs on a finite or infinite time horizons and we consider the pseudo-resolvent $R_r$ as in \eqref{eq:Rr}, 
and we assume that it fulfills the growth condition \eqref{arendtcond} for $\Omega=\{\lambda\in\C ~|~ \re\lambda> \omega\}$ for some $\omega<0$. Further, let $T$ be the associated exponentially stable degenerated semi-group. With the notation as in Proposition~\ref{prop:vdk}, we use the bounded projection $P$ which fulfills $\ran P=X_R$ and $\ker P=X_K$ and decompose $A^{-1}Bu=(I-P)A^{-1}Bu+PA^{-1}Bu=x_0+ x_1\in X$ and set $\hat B_0u:=(I-P)A^{-1}Bu$ and $\hat B_1u:=PA^{-1}Bu$ for all $u \in U$. Then $\hat B_1\in L(U,X_R)$, $\hat B_0\in L(U,X_K)$ and the mild solution of \eqref{adae} is given by 
\[
x(t)=T(t)x_0+\int_0^tT_{-1}(t-s){A_R}\hat B_1u(s)ds-\hat B_0u(t).
\]
Accordingly, the output is given by 
\begin{align}
\label{eq:output}
y(t)&=Cx(t)\\&=CT(t)x_0+C\int_0^tT_{-1}(t-s){A_R}\hat B_1u(s){\rm d}s-C\hat B_0u(t),
\nonumber
\end{align}
and it will be decomposed into the sum of the two following operators 
\begin{align*}
\Psi:X\rightarrow L^2([0,t_f],Y),\quad (\Psi x_0)(t):=C T(t)x_0,\, \text{$t\in[0,t_f]$}.
\end{align*}
By invoking the boundedness assumption \eqref{eq:GHinf}, we can, by using \cite{Weiss91}, infer that the {\em input-output operator}
\begin{align*}
\F:\; &L^2([0,t_f],U)\rightarrow L^2([0,t_f],Y),\\  (\F u)(t)&:=C\int_0^tT_{-1}(t-s){A_R}\hat B_1u(s)ds-C\hat B_0u(t).
\end{align*}
is a~bounded mapping. In summary, this means
\[
y=\Psi x_0+\F u.
\]

In the following, we consider the time-varying weights $\mathcal{Q}(\cdot)\in L^{\infty}([0,t_f],\R^{n\times n})$, $\mathcal{R}(\cdot)\in L^{\infty}([0,t_f],\R^{m\times m})$, which are assumed to be symmetric almost everywhere and $\mathcal{N}(\cdot)\in L^{\infty}([0,t_f],\R^{m\times n})$ which are used in the quadratic cost function
\begin{align*}
\hat J(u,y)&=\int_0^{t_f} \left\langle\begin{pmatrix}y(t)\\ u(t)\end{pmatrix},\begin{pmatrix}\mathcal{Q}(t) & \mathcal{N}^*(t)\\\mathcal{N}(t) & \mathcal{R}(t)\end{pmatrix}\begin{pmatrix}y(t)\\ u(t)\end{pmatrix} \right\rangle_{Y\times U}{\rm d}t\\&=\left\langle\begin{pmatrix}y\\ u\end{pmatrix},\begin{pmatrix}\mathcal{Q} & \mathcal{N}^*\\\mathcal{N} & \mathcal{R}\end{pmatrix}\begin{pmatrix}y\\ u\end{pmatrix} \right\rangle_{L^2}.
\end{align*}
Using \eqref{eq:output} the costs can be rewritten 
\begin{align*}
&J(u,x_0):=\hat J(u,\Psi x_0+\F u)\\&=\left\langle\begin{pmatrix}x_0\\ u\end{pmatrix},\begin{pmatrix}\Psi^*\mathcal{Q}\Psi & \Psi^*(\mathcal{Q}\F+\mathcal{N}^*)\\(\F^*\mathcal{Q}+\mathcal{N})\Psi & \mathscr{R}\end{pmatrix}\begin{pmatrix}x_0\\ u\end{pmatrix} \right\rangle
\end{align*}
with inner product in $X\times L^2([0,t_f],U)$ and the \emph{Popov operator}
\begin{equation}
\label{popovop}
\mathscr{R}:=\mathcal{R}+\mathcal{N}\F+\F^*\mathcal{N}^*+\F^*\mathcal{Q}\F\in L(L^2([0,t_f],U)).
\end{equation}

In the following we assume that the Popov operator \eqref{popovop} is \emph{coercive}, i.e.\ there exists $\varepsilon>0$ such that 
\begin{equation}
\label{coercive}\langle u,\mathscr{R}u\rangle\geq\varepsilon\|u\|^2,\quad \forall\,\text{$u\in L^2([0,t_f],U)$}.
\end{equation}

Using the coercivity assumption \eqref{coercive}, we can formulate the solution of the above optimal control problem.
\begin{prop}
\label{prop:3}
Assume that we have an abstract differential-algebraic system whose pseudo-resolvent $R_r$ is given by \eqref{eq:Rr} and satisfies \eqref{arendtcond}, \eqref{eq:GHinf} and assume that $\mathscr{R}$ satisfies the coercivity assumption \eqref{coercive}. Then for all $x_0\in X$
\[
\min_{u\in L^2([0,t_f],U)}J(u,x_0)=\langle x_0,\mathcal{P}x_0\rangle_X
\]
with the Riccati-operator $\mathcal{P}=\mathcal{P}^*\in L(X)$ given by 
\[
\mathcal{P}=\Psi^*\mathcal{Q}\Psi-\Psi^*(\mathcal{Q}\F+\mathcal{N}^*)\mathscr{R}^{-1}(\F^*\mathcal{Q}+\mathcal{N})\Psi
\]
and the unique minimizer 
\begin{equation}
\label{uopt}
u^{\rm opt}(t)=-\mathscr{R}^{-1}(\F^*\mathcal{Q}+\mathcal{N})\Psi(x_0)(t),\, \text{for a.e.\ $t\geq0$}.
\end{equation}
\end{prop}
\begin{proof}
We use an idea from \cite{WeisWeis97}, abbreviate $\hat u:=u+\mathscr{R}^{-1}(\F^*\mathcal{Q}+\mathcal{N})\Psi(x_0)$ and complete the square as follows
\begin{align*}
J(u,x_0)=\langle \mathcal{P} x_0,x_0\rangle_X+\langle\mathscr{R}\hat u,\hat u\rangle_{L^2([0,t_f],U)}.
\end{align*}
Since $\mathscr{R}$ is uniformly positive the minimum is attained if and only if the second term vanishes.\hfill $\qed$
\end{proof}

The formula \eqref{uopt} for the unique minimizer can be rewritten with  $y(t)=\Psi(x_0)(t)+\F u^{\rm opt}(t)$ as
\[
(I-\mathscr{R}^{-1}(\F^*\mathcal{Q}\F+\mathcal{N}\F))u^{\rm opt}(t)=-\mathscr{R}^{-1}(\F^*\mathcal{Q}+\mathcal{N})y(t).
\]

To obtain $u^{\rm opt}(\cdot)$ as an output feedback, we have to invert the operator $(I-\mathscr{R}^{-1}(\F^*\mathcal{Q}\F+\mathcal{N}\F))$. 
\begin{cor}
\label{cor:R}
In addition to the assumptions of Proposition~\ref{prop:3}, let the cost function $J$ satisfy $\mathcal{N}=0$ and let there exists $\varepsilon>0$ such that $\langle \mathcal{R}u,u\rangle\geq \varepsilon\|u\|^2$ holds for all $u\in L^2([0,t_f],U)$. Then 
\begin{align}
\label{eq:u_opt_special_case}
u^{\rm opt}(t)=-(I-\mathscr{R}^{-1}\F^*\mathcal{Q}\F)^{-1}\mathscr{R}^{-1}\F^*\mathcal{Q}y(t).
\end{align}
\end{cor}
\begin{proof}
With these assumptions the Popov operator satisfies \eqref{coercive} and we have
\[
\|\mathscr{R}^{-1}\F^*\mathcal{Q}\F\|\leq\frac{\|\F^*\mathcal{Q}\F\|}{\|\mathcal{R}+\F^*\mathcal{Q}\F\|}\leq \frac{\|\F^*\mathcal{Q}\F\|}{\varepsilon+\|\F^*\mathcal{Q}\F\|}<1.
\]
Therefore the inverse of $I-\mathscr{R}^{-1}\F^*\mathcal{Q}\F$ can be obtained from the Neumann series. \hfill$\qed$
\end{proof}

If the coercivity condition \eqref{coercive} is violated then the existence of an optimal control input $u^{\rm opt}$ is not guaranteed, which can already be observed for finite-dimensional DAEs \cite{ReisVoig19}. 

\begin{remark}[Infinite Horizon]
If the degenerated semigroup is exponentially stable of $T$ then the following stability conditions hold
\begin{align}
\label{extstable}
\Psi\in L(X,L^2([0,\infty),Y)),\\ \F\in L(L^2([0,\infty),U),L^2([0,\infty),Y)).
\end{align}
In particular, the result of Proposition~\ref{prop:3} and Corollary~\ref{cor:R} continues to hold. Well-posed linear systems that satisfy~\eqref{extstable} with  $E=I$ are called externally stable in the literature. For ODEs it was shown that for exponentially stabilizable and detectable systems the conditions~\eqref{extstable} are equivalent to the exponential stability of the system (also called internal stability), see \cite{R93}. 
\end{remark}

\begin{remark}
The mild solutions that are considered here can be characterized using the mild solutions of the pseudo-resolvent equation \eqref{eq:pseudoDGL} with $R=R_r$. It was shown in \cite[Lemma~8.1]{GernReis23} that one can also rewrite the ADAE \eqref{adae} using the pseudo-resolvent $R=R_l=EA^{-1}$ in the following way
\begin{align}\label{eq:pseudoDGL_l}    
\frac{{\rm d}}{{\rm d}t}EA^{-1}z(t)=z(t)+f(t)\quad EA^{-1}z(0)=EA^{-1}z_0.
\end{align}
This rewriting does not change the inhomogeneity, but to guarantee the existence of solutions the initial value of the DAE is restricted to $x_0=A^{-1}z_0\in D(A)$.
\end{remark}

\begin{remark}
    Let us further discuss some possible generalizations which allow to get rid of the general assumptions \eqref{eq:resgrowth} and \eqref{eq:GHinf} on the growth of the right resolvent and the transfer function.
    \begin{enumerate}[(a)]
        \item Instead of \eqref{eq:resgrowth} and \eqref{eq:GHinf}, we can also assume the weaker condition that there exists some $M>0$, $\omega\in\R$, such that
\[\|(A-\lambda E)^{-1}E\|\leq \frac{M}{\lambda-\omega} \quad \text{ $\forall\,\lambda>\omega$}  \]
along with boundedness of the transfer function $G(s)=C(sE-A)^{-1}B$ on the shifted complex half-plane $\{s\in\C:\re(s)>\omega\}$. Namely, by invoking that, $(x,u,y)$ fulfills \eqref{system} if, and only if, 
\[(x_\omega,u_\omega,y_\omega):=e^{-\omega\cdot}(x,u,y)\]
fulfills
\[\begin{split}
\tfrac{{\rm d}}{{\rm d}t}Ex_\omega(t)&=(A-\omega E)x_\omega(t)+Bu_\omega(t),\\ y_\omega(t)&=Cx_\omega(t),\quad Ex_\omega(0)=Ex_0,
\end{split}
\]
we can equivalently minimize the cost functional
\begin{multline*}
J(u_\omega, y_\omega) \\= \int_0^{t_f}\!\!\!\! \left<\begin{pmatrix}y_\omega(t)\\ u_\omega(t)\end{pmatrix},\begin{bmatrix}e^{\omega t}\mathcal{Q}(t) & e^{\omega t}\mathcal{N}^*(t) \\ e^{\omega t}\mathcal{N}(t) & e^{\omega t}\mathcal{R}(t)\end{bmatrix}\begin{pmatrix}y_\omega(t)\\ u_\omega(t)\end{pmatrix} \right>_{Y\times U}\!\!\!\!\!\!\! {\rm d}t
\end{multline*}
subject to the above system.
\item To get rid of the condition that the transfer function has to be bounded on some half-plane, and the right resolvent has to be bounded on some real half-line, one can impose the weaker condition that there exists some state feedback leading to a~system with this property. That is, there exists some $F\in L(X,U)$, such that, there exists some $M>0$, $\omega\in\R$, such that
\[
\|((A+BF)-\lambda E)^{-1}E\|\leq \frac{M}{\lambda-\omega} \quad \text{ $\forall\,\lambda>\omega$}    
\]
along with boundedness of $C(sE-(A+BF))^{-1}B$ 
on some half-plane. Namely, by invoking that, $(x,u,y)$ fulfills \eqref{system} if, and only if, 
\[(x_F,u_F,y_{F1},y_{F2}):=(x,u-Bx,y,Bx)\]
fulfills
\[\begin{split}
\tfrac{{\rm d}}{{\rm d}t}Ex_F(t)&=(A+BF)x_F(t)+Bu_F(t),\\ y_{F1}(t)&=Cx_F(t),\quad Ex_\omega(0)=Ex_0,\\
y_{F2}(t)&=Fx_F(t),
\end{split}
\]
we can equivalently minimize the cost functional
\begin{multline*}
J(u_F, y_{F1}, y_{F2}) \\= \int_0^{t_f}\!\!\!\! \left<\left(\begin{smallmatrix}y_{F1}(t)\\y_{F2}(t)\\ u_F(t)\end{smallmatrix}\right),\left[\begin{smallmatrix}\mathcal{Q}(t) & 0 &\mathcal{N}^*(t) \\ 0 & 0&-B\\\mathcal{N}(t)&-B^*&\mathcal{R}(t)\end{smallmatrix}\right]\left(\begin{smallmatrix}y_{F1}(t)\\y_{F2}(t)\\ u_F(t)\end{smallmatrix}\right) \right>_{Y\times U^2}\!\!\!\!\!\!\! {\rm d}t
\end{multline*}
subject to the above system. In the finite-dimensional case, the existence of a~state feedback with such properties is guaranteed, if the system is {\em impulse controllable}, see \cite{BerRei13}. 
\end{enumerate}
\end{remark}

\section{Application to heat-diffusion systems}
A standard heat diffusion system on a one-dimensional spatial domain $[0,L]\subseteq\R$ with Dirichlet boundary conditions, is described by 
\begin{align}
\label{eq:heat_diff}    
\frac{\partial T}{\partial t}(t,\xi) &= k \alpha \frac{\partial^2 T}{\partial \xi^2}(t,\xi)+\mathbf{1}_{I_U}(\xi)u(t), \\
 T(t,0)&=T(t,L)=0,\quad \text{for all $t\geq 0$,}
\end{align}
where $T$ is the temperature, $\alpha > 0$ is the diffusivity constant, $k>0$ is the thermal conductivity, $u\in L^2$ is a control input and $\mathbf{1}_{I_U}$ is the indicator function of a non-empty subinterval $I_U\subseteq[0,L]$. We choose $\mathcal{R}=I$, $\mathcal{Q}=I$, and $\mathcal{N}=0$, then $\mathscr{R}$ is coercive.

Following \cite{MehrZwar23}, we put \eqref{eq:heat_diff} in the ADAE framework, by 
considering the defining relation between $T$ and the heat flux $J$ and Fourier's law
\begin{equation*}
\frac{\partial T}{\partial t} = -\alpha \frac{\partial J}{\partial \zeta},\quad J = -k \frac{\partial T}{\partial \zeta}.
\end{equation*}
Then the equivalent ADAE $\frac{\mathrm{d}}{\mathrm{d}t}Ex=Ax$ reads
\begin{align*}
\frac{\mathrm{d}}{\mathrm{d}t}
\begin{bmatrix}
\alpha^{-1}I & 0 \\
0 & 0
\end{bmatrix}
\begin{bmatrix}
T \\
k^{-1}J
\end{bmatrix}
&=
\begin{bmatrix}
0 & -\tfrac{\partial}{\partial \zeta} \\
-\tfrac{\partial}{\partial \zeta} & -I
\end{bmatrix}
\begin{bmatrix}
T \\
k^{-1}J
\end{bmatrix}+\begin{bmatrix}
    0\\ \mathbf{1}_{I_U}
\end{bmatrix}u(t)\\
y(t)&=\mathbf{1}_{I_Y}\alpha E\begin{bmatrix}
T \\
k^{-1}J
\end{bmatrix}=\mathbf{1}_{I_Y}(\cdot)T(t,\cdot)
\end{align*}
for some open interval $I_Y\subseteq [0,L]$ and with the operator $A$ that is given by 
\[
A=\begin{bmatrix}
    0&\mathcal{D}^*\\ -\mathcal{D}&-I
\end{bmatrix},\quad D(\mathcal{D})=\{x\in H^1([0,L])\,:\, x(0)=0\}.
\]
Hence, by \cite[Lemma 2.7]{GernHins22}, the inverse $\lambda E-A$ exists for all $\lambda\in\mathbb{C}$ with $\re\lambda\geq -\omega$ for some $\omega>0$ and the pseudo-resolvent $R_r$ is given by 
\begin{align*}
R_r(\lambda)&=(A-\lambda E)^{-1}E\\&=\begin{bmatrix} -I
    &\mathcal{D}^*\\-\mathcal{D}& -\tfrac{\lambda}{\alpha}
\end{bmatrix}\begin{bmatrix}
    (\tfrac{\lambda}{\alpha}+\mathcal{D}^*\mathcal{D})^{-1}&0\\0&    (\tfrac{\lambda}{\alpha}+\mathcal{D}\mathcal{D}^*)^{-1}
\end{bmatrix}E\\
&=\begin{bmatrix}
    - (\lambda+\alpha\mathcal{D}^*\mathcal{D})^{-1}&0\\- \mathcal{D}(\lambda+\alpha\mathcal{D}^*\mathcal{D})^{-1}&0
\end{bmatrix},
\end{align*}
where $\mathcal{D}^*=-\mathcal{D}$ and $\mathcal{D}\mathcal{D}^*=-\tfrac{\partial^2}{\partial \xi^2}$ with domain $H_0^2([0,L])$. Therefore, the resolvent growth assumption~\eqref{arendtcond} and \eqref{eq:GHinf} hold. Therefore, the unique LQ optimal control is given by \eqref{eq:u_opt_special_case}.  

\section{Outlook}
One interesting application of the presented approach is using the particular cost $\mathcal{Q}=0$, $\mathcal{R}=0$ and $\mathcal{N}=I$ which leads to the required supply $J(x_0,u)=2\int_0^T\re \langle y(t),u(t)\rangle {\rm d}t$ and provides a link to passivity. The Popov operator is then given by $\mathscr{R}:=\F+\F^*$. 
Furthermore, one can investigate possible relaxation of the assumption on the regularity and the index of the underlying differential-algebraic equation. For finite dimensional spaces $X$, $U$ and $Y$ there is the problem whether there exists a state feedback $F\in L(X,U)$ or output feedback $F\in L(Y,U)$ such that the closed loop pair $(E,A+BF)$ is regular with index one. For ordinary systems it is known that such a feedback exists if and only if a finite cost condition holds and it is expected that these results can be generalized to abstract differential-algebraic equations.

\bibliographystyle{plain}
\bibliography{references}

\end{document}